\documentclass[12pt,a4paper]{amsart}

\usepackage{amsmath,mathrsfs}
\usepackage{amssymb}
\usepackage{color}
\usepackage{dsfont}
\usepackage[latin1]{inputenc}
\usepackage{tikz}
\usepackage[font=footnotesize]{caption}

\newcommand\NoBlackBoxes{\global\overfullrule0pt}
\NoBlackBoxes

\makeatletter
\let\serieslogo@\relax
\let\@setcopyright\relax

\newtheorem{definition}{Definition}[section]
\newtheorem{theorem}[definition]{Theorem}
\newtheorem{lemma}[definition]{Lemma}
\newtheorem{proposition}[definition]{Proposition}
\newtheorem{rem}[definition]{Remark}

\newtheorem{corollary}[definition]{Corollary}

\renewcommand{\P}{{\mathbb{P}}}
\newcommand{\E}{{\mathbb{E}}}
\newcommand{\R}{{\mathbb{R}}}

\renewcommand{\epsilon}{\varepsilon}
\renewcommand{\phi}{\varphi}
\newcommand{\vep}{\varepsilon}

\usepackage{ulem}

\begin{document}

\setcounter{page}{1}

\title[A Comparative Study of Sparse Associative Memories]{A Comparative Study of Sparse Associative Memories}

\author[Vincent Gripon]{Vincent Gripon}
\address[Vincent Gripon]{Telecom Bretagne
 UMR CNRS Lab-STICC,
 Technopole Brest Iroise,
 29238 Brest,
France}

\email[Vincent Gripon]{vincent.gripon@telecom-bretagne.eu}

\author[Judith Heusel]{Judith Heusel}
\address[Judith Heusel]{Fachbereich Mathematik und Informatik,
University of M\"unster,
Einsteinstra\ss e 62,
48149 M\"unster,
Germany}

\email[Judith Heusel]{jheus01@uni-muenster.de}

\author[Matthias L\"owe]{Matthias L\"owe}
\address[Matthias L\"owe]{Fachbereich Mathematik und Informatik,
University of M\"unster,
Einsteinstra\ss e 62,
48149 M\"unster,
Germany}

\email[Matthias L\"owe]{maloewe@math.uni-muenster.de}

\author[Franck Vermet]{Franck Vermet}
\address[Franck Vermet]{Laboratoire de Math\'ematiques, UMR CNRS 6205, Universit\'e de Bretagne Occidentale,  6, avenue Victor Le Gorgeu\\
CS 93837\\
F-29238 BREST Cedex 3\\
France}

\email[Franck Vermet]{Franck.Vermet@univ-brest.fr}

%\thanks{Research of the second author was supported by ??}

\date{\today}



%\newcommand{\wlim}{\mathop{\hbox{\rm w-lim}}}
%%\newcommand{\sgn}{\mathop{\hbox{\rm sgn}}}
%\newcommand{\na}{{\mathbb N}}
%\newcommand{\re}{{\mathbb R}}
%
%\newcommand{\vep}{\varepsilon}
%

%\begin{document}
%temporarily
%\today
%\tableofcontents
%temporarily
\begin{abstract}
We study various models of associative memories with sparse information, i.e. a pattern to be stored is a random string of $0$s and $1$s with about $\log N$ $1$s, only. We compare different synaptic weights, architectures and retrieval mechanisms to shed light on the influence of the various parameters on the storage capacity.
\end{abstract}

\maketitle

\section{Introduction}
Starting with the seminal paper \cite{griponb}, Gripon, Berrou and coauthors revived the interest in associative memory models, see e.g. \cite{griponc}, \cite{gripond}, \cite{gripone}, \cite{griponf}. Their approach is motivated by both biological considerations and ideas from information theory and leads to a neural network that is organized in clusters of interacting neurons. They state that their model (which we will refer to as the GB model) is more efficient (see \cite{griponb}) and has by far a larger storage capacity than the benchmark model for associative memories, the Hopfield model introduced in \cite{Hopfield1982}. Indeed, their considerations lead to a storage capacity of the order $N^2/(\log N)^2$ messages (or patterns or images; these words will be used synonymously) for their model with $N$ neurons, while the standard Hopfield model with N neurons only has a capacity of $N/(2 \log N)$ (see \cite{MPRV}, \cite{Bov98}).

However, the standing assumption of the GB model is that for $N$ neurons there are $c$ clusters of neurons with  $1\le c \le \log N$, and each message to be stored has only exactly one active neuron per cluster. This not only leads to a restriction on the number of storable messages, but also to them being very sparse (where sparsity is defined by a small number of active neurons).
As a matter of fact, for sparse messages other models of associative memories have been proposed by Willshaw \cite{Willshaw}, Amari \cite{Amari1989}, Okada \cite{Okada1996},  or \cite{Bolle}, \cite{LV_BEG}, and \cite{HLV15}.
All these models have in common that their storage capacity is conjectured to be much larger than that of the Hopfield model. The Willshaw model has also been discussed in a number of papers by Palm, Sommer, and coauthors (\cite{Palm1980}, \cite{Palm1996}, \cite{Palm2013} e.g.), with the difference that there the focus is rather on information capacity than on exact retrieval (and that many of the techniques are not rigorous).
In \cite{LV_BEG} it has been rigorously proven for a sparse three-state network, the so called Blume-Emery-Griffiths model, that the capacity is indeed of the predicted order (even though there, strictly speaking the degree of sparsity is not allowed to depend on the number of neurons).

A natural question is thus to separate the various factors that can influence the storage capacity of a model: the sparseness of the messages, the storage mechanism, and the algorithm to retrieve the stored patterns. The objective of the present article is to analyze this question. To this end we will try to give bounds on the storage capacity of the Willshaw model, Amari's version of a sparse 0-1 Hopfield model, and the GB model. In particular, we will see that all these models achieve a storage capacity of the order of $N^2/(\log N)^2$ when the number of active neurons $c$ satisfies $c = a \log N$ for some positive $a$. Also we will discuss the influence of model specificities to the absolute constants in the storage capacities.

More precisely, we organize our article in the following way. In the next section, we describe the three models we aim at studying and formally define what is meant by ``storing a message''.
In Section 3 we give some insight why an order of $N^2/(\log N)^2$ for the number of stored messages is to be expected in a model with $N$ neurons, of which about only $\log N$ are active. To this end we consider a certain event in the GB model that implies that a message cannot be retrieved correctly. In the fourth section we state our main results. These are proved in Section 5. Section 6 takes up ideas from Section 3 to show, that if the number of messages is too large, an erased message cannot be completed correctly in the GB model. Finally, Section 7 discusses some dynamical properties of the considered models and contains some simulations, in particular on the probability to correct an error in several steps. These probabilities are notoriously difficult to access analytically (see e.g. \cite{Burshtein}, \cite{LV07}, or \cite{LV15}).
The simulations give an impression of the advantages and drawbacks of the several models.

{\bf Acknowledgement:} We are very grateful to two anonymous referees for a very careful reading of a first version of the manuscript and valuable remarks that helped to improve its readability significantly.

\section{The models}
We will now present the models that are in the center of our interest in the present paper. The reference model is always the Hopfield model on the complete graph (i.e. all neurons are interconnected), with $M$ patterns $(\xi^\mu)_{\mu=1, \ldots M}=(\xi_i^\mu)_{i=1, \ldots N}^{\mu=1, \ldots M}\in\{ -1, +1\}^{N\times M}$. Here the so called synaptic efficacy $J_{ij}$ is given by
$$
J_{ij}=\sum_\mu  \xi_i^\mu\xi_j^\mu \qquad 1 \le i\neq j \le N
$$
and an input $\sigma \in \{ -1, +1\}^N$ is transformed by the dynamics
$$
T_i(\sigma)= \mathrm{sgn}(\sum_{j \neq i} J_{ij}\sigma_j)
$$
where $\mathrm{sgn}$ is the sign function (and the sign of 0 is chosen at random). This update can happen either synchronously or asynchronously in $i$. In \cite{MPRV} it was shown that for unbiased and i.i.d. random variables $((\xi_i^\mu)_{i=1, \ldots N})_{\mu=1, \ldots M}$ and  $M= c \frac N{\log N}$ with $c < \frac 12 $, an arbitrary message is stable under the dynamics with a probability converging to one. Of course, this model can be generalized to i.i.d. biased patterns with expectation $a$. In \cite{Lo_biased} the author suggests to replace the synaptic efficacy by $J_{ij}=\sum_\mu  (\xi_i^\mu-a)(\xi_j^\mu-a)$ and shows that the storage capacity (in the sense that an arbitrary pattern is a fixed point of the above dynamics) decreases for a strong bias. More precisely, he gives a lower bound on the storage capacity of the Hopfield model with biased patterns of the form $C p^2(1-p)^2 N/\log N$, where $C$ is an explicit constant that depends on the notion of storage capacity used and $p$ is the probability that $\xi_1^1$ equals $+1$. Note that this behaviour is amazingly similar to the behaviour of Hopfield models with correlated patterns, cf. \cite{L98}. Another model for biased $\pm 1$-patterns was proposed by Okada \cite{Okada1996}.

However, if we think of the bias as a certain sparsity of the patterns, it may be more natural to consider patterns $(\xi^\mu)_{\mu=1, \ldots M}$ where the $(\xi_i^\mu)$ still are i.i.d. but take values $0$ and $1$ where $\P(\xi_i^\mu=1)=p$ is small. We will henceforth consider such patterns and three such models.

\subsection{Amari's model}
The model Amari proposed in \cite{Amari1989} is closest in spirit to the Hopfield model. Here we take $J_{ij}=\sum_\mu \xi_i^\mu\xi_j^\mu$ and with this new setting, we consider input spin configurations $\sigma \in \{0,1\}^N$ and map their spins to either $0$ or $1$ with the help of a dynamics. Of course, one should only map an input spin $\sigma_i$ to 1, if the so called local field $\sum_{j \neq i} J_{ij}\sigma_j$ is large enough, say larger than a given threshold.
To compare Amari's results to the other models we choose
$$
\P(\xi_i^\mu=1)=p = \frac{\log N}N.
$$
As a matter of fact, this is the case of extremely diluted patterns, since if $p$ is even smaller, say $p=c/N$ for some $c$, with positive probability some of the patterns will entirely consist of $0$'s and will thus be indistinguishable.

We propose the following dynamics, where a spin $\sigma_i$ will be 1, if the so called local field
$$
S_i(\sigma)=\sum_{j \neq i} J_{ij}\sigma_j
$$
is large enough, say larger than a given threshold.
$$
T_i(\sigma)= \Theta(S_i(\sigma)-h)
$$
where $\Theta(x)= {\bf 1}_{\{x \ge 0\}}$ and we choose $h= \gamma \log N$ for some $\gamma >0$. Note that this seems a reasonable choice if we want the $(\xi^\mu)$ to be fixed points of the dynamics. Consider for example the case $\xi_i^1=1$ we have that
$$\sum_{j \neq i} J_{ij}\xi^1_j= \sum_{j\neq i} \xi_j^1+ \sum_{\mu \neq 1} \sum_{j\neq i} \xi_i^\mu \xi_j^\mu \xi_j^1$$
and the first term on the right hand side is of order $\log N$. Also note that Amari just considers the case of a fixed number $\log N$ of active neurons per message (which is similar), and states that the above model would perform much worse in the case we consider. We will see that this is not the case.

\subsection{The Willshaw model}
The following model was proposed in a celebrated paper by Willshaw \cite{Willshaw}. It corresponds to Amari's model with the restriction that the efficacy $J_{ij}$ does not depend on the {\sl number} of messages that use neurons $i$ and $j$ but just on whether there is any $\mu$ with $\xi_i^\mu \xi_j^\mu=1$. In the case of the Hopfield model this procedure is known as ``clipped'' synapses.

Formally, we will now either assume that the $(\xi_i^\mu)$ are i.i.d $0-1$ random variables with success probability $p=\frac{\log N}N$ or we take the $M$ messages to be realized uniformly at random from all sets of $M$ messages with exactly $c=\log N$ active neurons. Both cases are similar, but the first one is mathematically more convenient, because in this case the images as well as all their spins are independent. Moreover, in the Willshaw model we choose
$$
J_{ij}=\Theta(\sum_\mu  \xi_i^\mu\xi_j^\mu-1)=\begin{cases}
1 & \mbox{if } \exists \mu: \xi_i^\mu \xi_j^\mu=1 \\
0 & \mbox{otherwise,}
\end{cases}
$$
for all $i,j\in\{1,\ldots,N\}$.
There are two different (yet similar) types of dynamics to be considered. The first one is the threshold dynamics also considered in Amari's model. So again for an input $\sigma \in \{0,1\}^N$ we set
$$
T_i(\sigma)= \Theta(\bar{S}_i(\sigma)-h)
$$
with $\bar{S}_i(\sigma)=\sum_{j} J_{ij}\sigma_j$ and $h= \gamma {\log N}$, for some $\gamma >0$. This dynamics is applicable to both types of patterns (i.i.d. random variables $(\xi_i^\mu)$ or randomly chosen messages amongst all sets of $M$ messages with exactly $c$ active neurons). For the Willshaw model, we consider $\bar{S}_i$ instead of ${S}_i(\sigma)=\sum_{j\ne i} J_{ij}\sigma_j$, since simulations support that it improves performance to modify $S_i$ in order to account for self influence of neurons. This modification is well known and will be referred as ``memory effect''.

In the latter case of exactly $c$ active neurons per message and the messages being randomly chosen messages amongst all sets of $M$ messages with exactly $c$ active neurons there is another retrieval dynamics that requires the knowledge of all the $\bar{S}_i(\sigma)$ for $1 \le i \le N$. In this setting, for a given input $\sigma \in \{0,1\}^N$ we compute all the $\bar{S}_i(\sigma)$ and order them: they will be denoted by $h_{(1)} \ge h_{(2)} \ge \ldots \ge h_{(c)} \ge \ldots \ge h_{(N)}$. Then we set all neurons $i$ with $\bar{S}_i(\sigma) \ge h_{(c)}$ to 1 and the others to 0. Note that in case of a tie we may obtain more than $c$ 1's after a step of the dynamics. This procedure was called ``Winner takes all''-algorithm (WTA algorithm, for short) in \cite{gripong} in a model that is closely related to the following cluster model.

Similarly, we may as well imagine that $c$ is fixed but we do not know it.  In this case we could just take the most active neurons, i.e. set all neurons with a value $\bar{S}_i(\sigma)$ lower than $h_{(1)}$ to 0. Interestingly, if we consider as input a partially erased version  $\tilde\xi^\mu$ of a stored message $\xi^\mu$, for the one step retrieval we consider theoretically in Sections 4 and 5, this does not change anything as long as we consider the memory effect described above, since in this case $h_{(1)}=h_{(c)}$. This is because $h_{(1)}$ cannot be larger than the number of initial 1's in the dynamics input and this upper bound is reached for at least all the neurons that are active in the message $\xi^\mu$ we are looking for. Considering the performance of the model with several steps of the retrieval dynamics numerically, however, shows that the above threshold $h_{(c)}$ is superior to a threshold $h_{(1)}$.
As a matter of fact, the dynamics using $h_{(1)}$ as threshold does not benefit from performing more than one iteration (see Theorem~\ref{thm:stuck1iteration}). %% Indeed, either the correct message is retrieved, or at least one spurious neuron is set to 1. In the latter case, either convergence is reached (this is when the spurious neurons are connected to all the correct ones) or the next step in the dynamics leads back to the input case (the only neurons connected to all neurons set to 1 are the initial ones).

On the other hand using $h_{(c)}$ allows for improvement over the time.
Also note that the WTA algorithm with $h_{(1)}$ as threshold can be applied in the case where the $(\xi_i^\mu)$ are i.i.d $0-1$ random variables with success probability $p=\frac{\log N}N$, as will be proven in Section 4.

\subsection{The GB model}
Here we assume that $N= l \log l =: l \cdot c$ for some $l$.
One tries to store $M$ messages $\xi^1,\ldots,\xi^M$ in a network with a block structure. The messages are sparse in the sense that each message $\xi^\mu$ has $c$ active neurons, only, one in each block of $l$ neurons.
To take into account the block structure, we will denote by $(a,k)$ the $k$-th neuron of the $a$-th block.

%I.e.$\xi^\mu=(\xi_1^\mu,\ldots,\xi_c^\mu)$, and for each $\mu=1, \ldots M$ and each $i=1, \ldots c$, $\xi_i^\mu$ denotes the (only) active neuron of the message $\xi_\mu$ in the i'th block.
%With such a message $\xi^\mu$ one associates the edges of the complete graph $\xi^\mu$ spanned by the vertices $\xi_1^\mu,\ldots,\xi_c^\mu$.
%A message $\xi^0=(\xi_1^0,\ldots,\xi_c^0)$ is considered to be stored in the model if all edges of the complete graph spanned by $(\xi_1^0,\ldots,\xi_c^0)$ are present in the set of edges $$\mathcal{E}:=\{e: e \mbox{ is an edge of one of the } m^\mu\}.$$

%The GB model can be mathematically described as follows:
%Set $\mathcal{A}= \{1, \ldots, l\}$. A message $\xi^\mu$ is then a string $\xi^\mu=(\xi_1^\mu,\ldots,\xi_c^\mu)\in \mathcal{A}^c$.
%With a message $\xi^\mu$  we associate a (column) vector $\psi(\xi^\mu)\in (\{0,1\}^l)^c$ obtained by replacing the $\xi_i^\mu$ with the unit vector $e_{\xi_i^\mu}$.
%Abusing notation we will also write $\mathcal{A}^c$ for the set $(\{0,1\}^l)^c$.
%Now build the 0-1-matrix
%$$ W = \max_{\xi \in \mathcal{M}} \psi(\xi) \psi(\xi)^t $$
%where $\mathcal{M}=\{\xi^1,\ldots,\xi^M\}$ and $\psi(\xi)^t$ is the transpose of $\psi(\xi)$.
For $a\ne a'$, an edge $e=((a,k),(a',k'))$ is said to be active for the message $\xi^\mu$ if
$\xi^\mu_{(a,k)} \xi^\mu_{(a',k')}=1$.
%$ W_{(a,k),(a,k')} = 1$.
%\item and $\tilde W_{(a,k),(a,k)} = 1$ if and only if there exists $\mu$ such that the k'th neuron in block $a$ is 1 (this is equivalent
%to adding a self-loop to the graph for each vertex $(a,k)$ such that there exists
%$\mu$ with $m^\mu_a= k$.
%\end{itemize}
Let $$\mathcal{E}((\xi^\mu)_{\mu=1,\ldots,M}):=\{e: e \mbox{ is an active edge of one of the } \xi^\mu\}.$$

\textit{We can also define the graph associated with or spanned by an arbitrary message $\xi^0$. This will be the (necessarily complete) graph with all vertices $(a,k)$ such that $\xi^0_{(a,k)}=1$ and edges $e=((a,k),(a',k'))$ for all $a,a',k,k', a\ne a'$ such that $\xi^0_{(a,k)} \xi^0_{(a',k')}=1$. Then a message $\xi^0$ is considered to be stored in the model if all edges of this complete graph spanned by $\xi^0$ are present in the set of edges $\mathcal{E}((\xi^\mu)_{\mu=1,\ldots,M})$.}

Similar to the Willshaw model, we define the synaptic efficacy by
$$W_{(a,k),(a',k')}= \Theta(\sum_{\mu=1}^M  \xi_{(a,k)}^\mu\xi_{(a',k')}^\mu-1).$$
Thus for $a\not=a'$ $ W_{(a,k),(a',k')} = 1$ if and only if $(a,k)$ and $(a',k')$ are activated simultaneously in one of the messages (both in the same message). On the other hand, for $a=a'$ we have $ W_{(a,k),(a,k')} = 1$ if and only if $k=k'$ and there exists $\mu$ such that the $k$'th neuron in block $a$ is 1. As a matter of fact, this description shows that the major difference to the Willshaw model is that in the GB model one has a restriction of the location of the 1's.

With this synaptic efficacy one can associate a dynamics $T$ on $(\{0,1\}^l)^c$ : instead of the local field $S_i(\sigma)$ of the preceding models, we define
$$
S_{(a,k)}(\sigma)=\sum_{b=1}^c \sum_{r=1}^l  \Theta(W_{(a,k),(b,r)}\sigma_{(b,r)}-1)
$$
for $\sigma\in (\{0,1\}^l)^c$, and the dynamics
$$
T_{(a,k)}(\sigma) =  \Theta(S_{(a,k)}(\sigma)- h).
$$
Here again $h$ is a threshold that needs to be adapted to the tasks we want the network to perform. E.g., choosing $h=c$ one readily verifies that all stored messages $\xi\in {\mathcal M}=\{\xi^1,\ldots,\xi^M\}$ are stable, i.e. we have $T(\xi)= \xi.$ Obviously, this can only go to the expense of error tolerance of the network.

The dynamics described above is the equivalent of the threshold dynamics in the Willshaw model. As in the latter model, we can also define a WTA algorithm. This will respect the local nature of the GB model.
To describe it, assume we want to update the values of the neurons in the $a$'th cluster $\sigma_{(a,k)}, k=1, \ldots,  l$. For each $k=1, \ldots,  l$ we then build
\begin{equation}
  s_{(a,k)}(\sigma)= \sum_{b=1}^c \Theta(\sum_{r=1}^l W_{(a,k), (b,r)} \sigma_{(b,r)} -1).
  \label{eq:som}
\end{equation}
(This is called the SUM-OF-MAX rule in \cite{griponh}; it accounts for the fact that in each message there only can be one active connection between two clusters). We then order the $s(a,k), k=1,\ldots,l$ and set the neuron(s) with the largest value to 1 and all others to 0.

\section{Wrong messages and a first bound on the storage capacity}

In this section we will approach the question: what could be the right order for the storage capacity of the above networks?

At first glance, storage capacity may refer to different properties of the network. E.g. from Section 4 we will ask ourselves: how many messages can we store such that they are fixed points of the network dynamics or how many messages can we register in our network such that even a certain number of errors can be corrected? On the other hand, in the previous section we already learned that in the GB model with a threshold dynamics,  an arbitrary number of input messages is stable if we choose the threshold equal to $c$, the number of active neurons. It is intuitively clear that this can only have a negative effect on the error retrieval abilities of the network, if we store too many messages in the network.

An extreme case of such a lack of error tolerance is if we recognize an input as a stored message even if it is not.
This property will be discussed in greater detail for the GB model and partially for the Willshaw model in this section. The insight we gain will provide us with an idea of how many messages we can store in the models.

We will prove the following theorem.

\begin{theorem} \label{Thsection3}
Consider the GB model with the threshold retrieval dynamics and threshold $h=c$. Take
$$
M=\alpha (\log c) l^2= \alpha l^2 \log \log l.
$$

If $\alpha>2$,  a random message (independent of the stored patterns) will be recognized as a stored message with probability converging to 1 as $l\rightarrow \infty$.

If $\alpha=2$ and as $l\rightarrow \infty$, with strictly positive probability a random message  will be recognized as a stored message.

On the other hand, if $\alpha < 2$ the probability that a random message will be recognized as stored goes to zero as $l\rightarrow \infty$.

\end{theorem}

We will use positive association of random variables (see e.g. \cite{Esary/Proschan/Malkup:1967}) to prove this theorem.
Recall that a set of real valued random variables $\mathbf{X}= (X_1, X_2, \ldots, X_n)$ is positively associated, if for any non-decreasing functions $f$ and $g$ from $\R^n$ to $\R$ for which the corresponding expectations exist we have
$$
\mathrm{Cov}(f(\mathbf{X}),g(\mathbf{X}))\ge 0.
$$
Also recall that independent random variables are positively associated and that non-decreasing functions of positively associated random variables remain positively associated.

For positively associated random variables we will repeatedly apply the following inequality.

\begin{lemma}(see \cite{Boutsikas2000}, Theorem 1)
Let $X_1, X_2, \ldots, X_n$ be positively associated integer valued random variables. Then
$$
0 \le  \P[X_i=0, i=1,\ldots,n]- \prod_{i=1}^n  \P[X_i=0] \le \sum_{1 \le i < j \le n} \mathrm{Cov}(X_i, X_j).
$$
\end{lemma}

\begin{proof}[Proof of Theorem \ref{Thsection3}]
Let $\xi^0$ be a random message. Without loss of generality we may (after relabelling) assume that $\xi^0_{(a,1)}=1$, for all $a=1,\ldots, c$.
Let ${\mathcal G}(\xi^0)$ be the event that $\xi^0$ is stored in the GB model. Its probability $\P({\mathcal G}(\xi^0))$ is given by
$$
\P({\mathcal G}(\xi^0))= \P(\forall a,b \in \{1, \ldots, c\},a \neq b, \exists \mu \in \{1, \ldots, M\}: \xi_{(a,1)}^\mu  \xi_{(b,1)}^\mu=1).
$$
 Note that the latter can be rewritten as
$$\P(\forall a,b \in \{1, \ldots, c\},a \neq b : \max_\mu \xi_{(a,1)}^\mu  \xi_{(b,1)}^\mu=1).$$
Now the $(\xi_{(a,1)}^\mu)$ are independent $0-1$-valued random variables, and taking their product and the maximum of these products are increasing functions of them. Thus $\{ \max_\mu \xi_{(a,1)}^\mu  \xi_{(b,1)}^\mu, a \neq b \}$ are positively associated (see e.g. \cite{Esary/Proschan/Malkup:1967}), which implies
\begin{eqnarray*}
\P(\forall a,b \in \{1, \ldots, c\},a \neq b: \max_\mu \xi_{(a,1)}^\mu  \xi_{(b,1)}^\mu=1) &\ge&  \P(\max_\mu \xi_{(a,1)}^\mu  \xi_{(b,1)}^\mu=1)^{c (c-1)/2}    \\
&=& (1- (1- 1/l^2)^M)^{c (c-1)/2}
\end{eqnarray*}
where on the right hand side of the above inequality $a$ and $b$ is an arbitrary pair of distinct variables.

Choosing $M=\alpha \log c l^2$ we see that the right hand side is approximately given by
$$
(1- (1- 1/l^2)^M)^{c (c-1)/2}   \approx \exp\left(-\frac{c^2}{2} e^{-\alpha \log c }\right)
$$
which converges to 1  if $\alpha >2$, and to $e^{-1/2}$ if $\alpha =2$.

On the other hand, we can also use positive association for an upper bound.
We put $X_e= \max\{ \xi_{(a,1)}^\mu  \xi_{(b,1)}^\mu,  \mu=1,\ldots,M\}$ for $e=((a,1),(b,1))$ and
$$
Z=\sum_{e\in V} X_e \quad \mbox{with  } V=\{((a,1),(b,1)),  a, b \in \{1, \ldots, c\}, a \neq b \}.
$$
Trivially,
$$
\P[{\mathcal G}(\xi^0)] = \P[Z=c (c-1)/2].$$
On the other hand,  the random variables $Y_e=1-X_e$ are also positively associated integer valued, and we may use the above lemma to arrive at
$$ \P[Z=L] \le \prod_e \P[Y_e=0]   +   \sum_{e,e'\in V} \mbox{Cov}(Y_e, Y_{e'})$$
i.e.
\begin{eqnarray} \label{exp_ineq}
 \P[Z=L]& \le&d^L   +   \sum_{e,e'\in V} \mbox{Cov}(X_e, X_{e'})
 \end{eqnarray}
where we set $d:=(1- (1- 1/l^2)^M)$ and we are left with computing the covariances.
To this end notice that $\mbox{Cov}(X_e, X_{e'})=0$, if $e$ and $e'$ are disjoint. So assume that $e=((a,1),(b,1))$ and $e'=((a,1),(b',1))$ and put
$\mathcal{M}(a,1):= \{\mu: \xi_{(a,1)}^\mu =1\}$ .
Then
\begin{eqnarray*}
\E(X_e X_{e'}) &=&   \P(\exists \mu, \nu \in \mathcal{M}(a,1): \xi_{(b,1)}^\mu=1, \xi_{(b',1)}^\nu=1)\\
&=&   \sum_{r=0}^M   \P(\exists \mu, \nu \in \mathcal{M}(a,1):\xi_{(b,1)}^\mu\xi_{(b',1)}^\nu=1|\, |\mathcal{M}(a,1)|=r) \P(|\mathcal{M}(a,1)|=r)\\
&=&   \sum_{r=0}^M   \P(\exists \mu \in \mathcal{M}(a,1): \xi_{(b,1)}^\mu=1|\, |\mathcal{M}(a,1)|=r)^2   \P(|\mathcal{M}(a,1)|=r)\\
&=&   \sum_{r=0}^M (1-(1-1/l)^r)^2 \binom M r (1/l)^r (1-1/l)^{M-r},
 \end{eqnarray*}
as on $\mathcal{M}(a,1)$ the events
$$
\{\exists \mu \in \mathcal{M}(a,1): \xi_{(b,1)}^\mu=1\}\quad \mbox{and  }   \{\exists \nu \in \mathcal{M}(a,1): \xi_{(b',1)}^\nu=1\}
$$
are independent and have equal probabilities.
The expression on the right hand side can be simplified to give
\begin{eqnarray*}
\E(X_e X_{e'}) &=& 1- 2  \sum_{r=0}^M  \binom M r (1/l)^r (1-1/l)^{M}+ \sum_{r=0}^M  \binom M r (1/l)^r (1-1/l)^{M+r}\\
&=&     1- 2(1-1/l)^M  (1+1/l)^M +   (1-1/l)^M(1+\frac{1}{l} (1-1/l))^M\\
&=&     1- 2(1-1/l^2)^M +\left(1-2/l^2+ 1/l^3\right)^M.
\end{eqnarray*}

On the other hand,
$$
(\E(X_e))^2=(\P(X_e=1))^2= d^2= (1-(1-\frac1{l^2})^M)^2.
$$
This yields
\begin{eqnarray*}
\mbox{Cov}(X_e, X_{e'}) &=&     1- 2\left(1-1/l^2\right)^M +\left(1-2/l^2+ 1/l^3\right)^M- \left(1-\left(1-1/l^2\right)^M\right)^2\\
 &=&    \left(1-2/l^2+ 1/l^3\right)^M -   \left(1-2/l^2+ 1/l^4\right)^M \\
 &=&  \exp\Bigl(M \log \bigl(1-2/l^2+ 1/l^3\bigr)\Bigr)- \exp\Bigl(M \log \bigl(1-2/l^2+ 1/l^4\bigr)\Bigr)\\
  &=&\exp\left(-2M/l^2\right) \left(M/l^3+ \mathcal{O}\left(M/l^4\right)\right),
\end{eqnarray*}
after expanding the logarithm and the exponential and taking into account that $M(\frac 1l)^3$ converges to 0 for our choice of the parameters.
Thus for
$$
M=\alpha (\log \log N) N^2/(\log N)^2
$$
we obtain because of $c= \log l \approx \log N$.
\begin{eqnarray*}
\sum_{e,e'\in V} \mbox{Cov}(X_e, X_{e'}) &\le& \alpha (\log \log N) c^4 \exp(-2\alpha \log \log N)/N\\
 &\approx & \frac 1 N  \alpha (\log \log N) (\log N)^4 \exp(-2\alpha \log \log N)
\end{eqnarray*}
Inserting this into \eqref{exp_ineq},
we obtain
\begin{eqnarray*}
 \P[{\mathcal G}(\xi^0)] &=& \P[Z=c (c-1)/2]    \\
&\le & d^L   +  \sum_{e,e'\in V} \mbox{Cov}(X_e, X_{e'}) \\
&\le &  d^L + \frac 1 {N} \alpha (\log \log N) (\log N)^4 \exp(-2 \alpha \log \log N)\\
&\le &  d^L + \frac 1 {N} \alpha  (\log N)^{(4-2 \alpha)} \log \log N
\end{eqnarray*}
The second summand on the right hand side clearly vanishes. But also $d^L$ converges to 0 for $\alpha <2$ (which can be seen as in the first part of the proof).
 Thus $\P[{\mathcal G}(\xi^0)]$ converges to 0,  and we can remark that  $\P[{\mathcal G}(\xi^0)]$ is exactly of order $d^L$ for $\alpha\in ]1,2[$.
 \end{proof}

\begin{rem}
\normalfont
The above computation also justifies a choice of $c$ that is not of constant order. Indeed, for $c$ being a constant independent of $N$ the same approximation of  $\P[{\mathcal G}(\xi^0)]$ by $d^L$ is true. However
$d^L$ converges to a constant larger than 0, even if $M=l^2$.
\end{rem}

A very similar theorem holds true, for the Willshaw model with an intensity of $1$s given by $\P(\xi_i^\mu=1)= \frac{\log N}N$.

\begin{theorem}
Consider the Willshaw model with i.i.d. messages and coordinates such that $\P(\xi_i^\mu=1)= \frac{\log N}N$. Consider the threshold retrieval dynamics with threshold $h=c$.                                                                                                                                                                                                                             Take $M= \alpha \frac{N^2}{(\log N)^2} \log \log N$.

If $\alpha> 2$ a random message with $c$ active neurons (independent of the stored
patterns) will be recognized as a stored message with probability converging to 1 as $l\rightarrow \infty$.

If  $\alpha=2$ and as $l\rightarrow \infty$, with strictly positive probability a random message will be recognized as a stored message.

On the other hand, if $\alpha < 2$ the probability that a random message will be recognized as stored goes to zero as $l\rightarrow \infty$.
\end{theorem}

The proof is almost identical to the proof of the previous theorem. We therefore omit it.

\section{Stability and error correction}
In this section we will try to give lower and sometimes also upper bounds on the number of patterns we can store in the various models, such that the given messages are stable under the dynamics of the network and errors in the input can be corrected.

We saw that in the GB model and the Willshaw model, slightly more than $N^2/(\log N)^2$ already suffice to supersaturate the networks. We will therefore always assume that $M=\alpha N^2/(\log N)^2$.

We start with Amari's model.

\begin{theorem}  \label{amari}
Suppose that in Amari's model with threshold $h=\gamma \log N$ ($\gamma <1$ to be chosen appropriately),  we have that $M=\alpha N^2/(\log N)^2$. Then, if $\alpha <e^{-2}$ for any fixed $\mu$, we have
$$
\P(\forall i: T_i(\xi^\mu)=\xi_i^\mu) \to 1
$$
as $N \to \infty$.

Moreover, for any error rate $0 < \rho < 1$, if $\gamma < 1-\rho$ is chosen appropriately and $\alpha <(1-\rho)e^{-(1+\frac 1{1+\rho})}$,  for any fixed $\mu$,  and any $\tilde \xi^\mu$ obtained by deleting at random $\rho \log N$ of the 1's in $\xi^\mu$, we have:
$$
\P(\forall i:  T_i(\tilde \xi^\mu)=\xi_i^\mu) \to 1
$$
as $N \to \infty$.

Finally, if $M > -\log (1-e^{-1})N^2/(\log N)^2$
$$
\P(\forall i: T_i(\xi^\mu)=\xi_i^\mu) \to 0
$$
as $N \to \infty$.
\end{theorem}

It is interesting to observe that the previous theorem also gives a result on the Willshaw model with a threshold dynamics.

\begin{corollary} \label{willshaw1}
In the Willshaw model with i.i.d. random variables $\xi^\mu_i$, threshold $h=\gamma \log(N)$, $\gamma <1$ and $M=\alpha N^2/(\log N)^2$ for $\alpha <e^{-2}$ we have for any fixed $\mu$
$$
\P(\forall i: T_i(\xi^\mu)=\xi_i^\mu) \to 1
$$
as $N \to \infty$.

Moreover, for any error rate $0 < \rho < 1$, if $\gamma < 1-\rho$ is chosen appropriately and $\alpha
<(1-\rho)e^{-(1+\frac 1{1+\rho})}$,  for any fixed $\mu$,  and any $\tilde \xi^\mu$ obtained by deleting at random $\rho \log N$ of the 1's in $\xi^\mu$, we have:
$$
\P(\forall i:  T_i(\tilde \xi^\mu)=\xi_i^\mu) \to 1
$$
as $N \to \infty$.

Finally, if $M > -\log (1-e^{-1})N^2/(\log N)^2$
$$
\P(\forall i: T_i(\xi^\mu)=\xi_i^\mu) \to 0
$$
as $N \to \infty$.
\end{corollary}

In computer simulations the threshold dynamics in the Willshaw model is outperformed by WTA.
Our theoretical results are by now limited to the question of the stability of messages and one step
of the retrieval dynamics.

\begin{theorem}\label{willshaw2}
Consider the Willshaw model with i.i.d. messages and coordinates such that $\P[\xi_i^\mu=1]=\frac cN$, where $c=\log(N)$. Consider the WTA dynamics with threshold $h_{(1)}$ and let
$M=\alpha N^2/(\log N)^2$. Then for $\alpha<-\log(1-e^{-1})$ we have for any fixed $\mu$
$$
\P(\forall i: T_i(\xi^\mu)=\xi_i^\mu) \to 1
$$
as $N \to \infty$.

This bound is sharp: For $\alpha>-\log(1-e^{-1})$ we have for any fixed $\mu$
$$
\P(\exists i: T_i(\xi^\mu)\not=\xi_i^\mu) \to 1
$$
as $N \to \infty$.

Finally,  if $\rho \log N$, $0\le \rho <1$
of the initial 1's of message $\xi^\mu$ are erased at random to obtain $\tilde \xi^\mu$,
we can prove the following result:\\
For $\alpha<-\log(1-e^{-1/(1-\rho)})$ we have for any fixed $\mu$
$$
\P(\forall i: T_i(\tilde \xi^\mu)=\xi_i^\mu) \to 1
$$
as $N \to \infty$.

Again, this bound is sharp: For $\alpha>-\log(1-e^{-1/(1-\rho)})$ we have for any fixed $\mu$
$$
\P(\exists i: T_i(\tilde\xi^\mu)\not=\xi_i^\mu) \to 1
$$
as $N \to \infty$.
\end{theorem}

\begin{rem}
\normalfont
For mathematical convenience, we assumed in  Th.  \ref{willshaw2} that the stored messages are independent, with i.i.d. coordinates  $(\xi_i^\mu)$ such that  $$\P[\xi_i^\mu=1]=\frac cN.$$
We can naturally expect the same results in the case where exactly $c$ neurons are active in each stored message, but properties of independence are lacking to prove such results in this situation.
\end{rem}

A very similar statement holds for the GB model with the WTA algorithm.

\begin{theorem}\label{GB-WTA}
In the GB model with independent messages with WTA dynamics (which again is called $T$) let
$M=\alpha l^2/c^2$. Then for $\alpha<-\log(1-e^{-1})$ we have for any fixed $\mu$
$$
\P(\forall (a,c): T_{(a,c)}(\xi^\mu)=\xi^\mu_{(a,c)}) \to 1
$$
as $N \to \infty$.

If $\rho \log N$ of the initial 1's of a message $\xi^\mu$ are erased at random to  construct $\tilde \xi^\mu$,
we obtain:
For $\alpha<-\log(1-e^{-1/(1-\rho)})$ we have for any fixed $\mu$
$$
\P(\forall (a,c):  T_{(a,c)}(\tilde\xi^\mu)=\xi^\mu_{(a,c)}) \to 1
$$
as $N \to \infty$.

\end{theorem}

\section{Proofs}
This section contains the proofs of the results in the previous section.
We start with Theorem \ref{amari}.

\begin{proof}[Proof of Theorem \ref{amari}]
Recall the situation of the theorem.
We choose $h=\gamma\log(N)$ with $\gamma\in(0,1)$. Then, for each $\delta\in(0,1)$,
\begin{align*}
&\mathbb{P}(\exists 1\leq i\leq N, T_i(\xi^1)\neq\xi_i^1)\\
\leq&\mathbb{P}\big(\big\vert\log(N)-\sum_{j}\xi_j^1
\big\vert\ge (1-\delta )\log(N)\big)\\
+&\mathbb{P}\left(\Big\lbrace\exists 1\leq i\leq N, T_i(\xi^1)\neq\xi_i^1\Big\rbrace\text{ }\cap \Big\lbrace\Big\vert\log(N)-\sum_{j}\xi_j^1
\Big\vert<(1-\delta )\log(N)\Big\rbrace\right)
\end{align*}
and the first term disappears as ${N}\rightarrow\infty$ due to the law of large numbers, since the $\xi_i^\mu$ are Bernoulli random variables with success probability $p=\log N/N$.

Let $\delta>\gamma$. If
$$
\Big\vert\log(N)-\sum_j\xi_j^1\Big\vert<(1-\delta)\log(N),$$
we have that
$\sum_j\xi_j^1>\delta\log(N)$, and for each i with $\xi_i^1=1$, we obtain
$$
S_i(\xi^1)=\sum_{j\neq i}J_{ij}\xi^1_j=\xi_i^1\sum_{j\ne i}\xi_j^1+\sum_{j\ne i}\xi_j^1\sum_{\mu=2}^M\xi_i^\mu\xi_j^
\mu\geq\sum_{j\ne i}\xi_j^1\ge \delta\log(N) -1\ge \gamma \log(N),
$$
for $N$ large enough, i.e. $T_i(\xi^1)=1$.

On the other hand, for each $i$ with $\xi_i^1=0$, we get
\begin{align*}
&\mathbb{P}\left(\{T_i(\xi^1)\neq\xi_i^1\}\cap\{\xi_i^1=0\}\cap\Big\{\Big\vert\log(N)-\sum_{j}\xi_j^1
\Big\vert<(1-\delta )\log(N)\Big\}\right)\\
\leq&\sum_{k=\lfloor \delta\log(N)\rfloor}^{\lceil (2-\delta)\log(N)\rceil}\mathbb{P}
\left(\Big\{\sum_{j\neq i}\xi_j^1\sum_{\mu=2}^M\xi_i^\mu\xi_j^
\mu\geq \gamma\log(N)\Big\}\cap\{\xi_i^1=0\}\cap \Big\{\sum_{j}\xi_j^1
=k\Big\}\right)\\
=&\sum_{k=\lfloor \delta\log(N)\rfloor}^{\lceil (2-\delta)\log(N)\rceil}\mathbb{P}
\left(\Big\{\sum_{j}\xi_j^1\sum_{\mu=2}^M\xi_i^\mu\xi_j^
\mu\geq \gamma\log(N)\Big\}\cap\{\xi_i^1=0\}\Big\vert\text{   } \sum_{j}\xi_j^1
=k\right)\cdot\mathbb{P}\left(\sum_{j}\xi_j^1
=k\right),\\
\end{align*}
since for $j=i$, the term $\xi_j^1\sum_{\mu=2}^M\xi_i^\mu\xi_j^\mu$ is equal to 0.

This yields
\begin{align*}
&\mathbb{P}\left(\{T_i(\xi^1)\neq\xi^1_i\}\cap\{\xi_i^1=0\}\cap\Big\{\Big\vert\log(N)-\sum_{j}\xi_j^1
\Big\vert<(1-\delta )\log(N)\Big\}\right)\\
\leq&\max_{\lfloor \delta\log(N)\rfloor\leq k\leq \lceil (2-\delta)\log(N)\rceil}\mathbb{P}\left(\sum_j\xi_j^1
\sum_{\mu=2}^M\xi_i^\mu\xi_j^
\mu\geq \gamma\log(N)\Big\vert\text{   } \sum_{j}\xi_j^1
=k\right)\cdot\\
&\sum_{k=\lfloor \delta\log(N)\rfloor}^{\lceil (2-\delta)\log(N)\rceil}\mathbb{P}\left(\sum_{j}\xi_j^1
=k\right)\\
\leq&\mathbb{P}\left(\sum_j\xi_j^1\sum_{\mu=2}^M\xi_i^\mu\xi_j^
\mu\geq \gamma\log(N)\Big\vert\text{   } \sum_{j}\xi_j^1
=\lceil(2-\delta )\log(N)\rceil\right),
\end{align*}
since the quantity $\sum_j\xi_j^1\sum_{\mu=2}^M\xi_i^\mu\xi_j^\mu$ is increasing with  $\sum_j\xi_j^1$, and the maximum is attained for $k=\lceil(2-\delta )\log(N)\rceil$.

Without loss of generality, $(2-\delta)\log(N)\in\mathbb{N}$ and $\xi_j^1=1$, $1\leq j\leq (2-\delta)\log(N)$; $\xi_j^1=0$, $j> (2-\delta) \log(N)$. Then, for each $t>0$,
\begin{align*}
&\mathbb{P}\left(\sum_{j}\xi_j^1\sum_{\mu=2}^M\xi_i^\mu\xi_j^
\mu\geq \gamma\log(N)\Big\vert\text{   } \sum_{j}\xi_j^1
=(2-\delta )\log(N)\right)\\
=&\mathbb{P}\left(\sum_{j=1}^{(2-\delta)\log(N)}\sum_{\mu=2}^M\xi_i^\mu\xi_j^
\mu\geq \gamma\log(N)\right)\\
\leq&e^{-t\gamma\log(N)}\mathbb{E}\exp\left(t\displaystyle\sum_{j=1}^{(2-\delta)
\log(N)}\sum_{\mu=2}^M\xi_i^\mu\xi_j^\mu\right)\\
%\end{align*}
%\begin{align*}
=&e^{-t\gamma\log(N)}\left[\mathbb{E}\exp\left(t\sum_{j=1}^{(2-\delta)
\log(N)}\xi_i^2\xi_j^2\right)\right]^{M-1}\\
=&e^{-t\gamma\log(N)}\left(1-p+p\left(1-p+pe^t\right)^{(2-\delta)\log(N)}\right)^{M-1}\\
\leq&e^{-t\gamma\log(N)}\left(1-p+pe^{p(e^t-1)(2-\delta)\log(N)}\right)^{M-1}\\
\leq&\exp\left[-t\gamma\log(N)+(M-1)p\left(e^{p(e^t-1)(2-\delta)\log(N)}-1\right)\right]\\
=&\exp\left[-t\gamma\log(N)+(M-1)p\left(p(e^t-1)(2-\delta)\log(N)+\mathcal{O}(\log(N)p^2)\right)\right]\\
=&\exp\left[-t\gamma\log(N)+Mp^2(e^t-1)(2-\delta)\log(N)+\mathcal{O}(M\log(N)p^3)\right],
\end{align*}
using $1+u \le e^u$ for all $u\ge 0$, expanding the exponential and assuming $t$ to be small.

Assuming $M=\alpha N^2/\log(N)^2$, we obtain that the last line is equal to
\begin{align*}
&\exp\left[-t\gamma\log(N)+Mp^2(e^t-1)(2-\delta)\log(N)+\mathcal{O}(M\log(N)p^3)\right]\\
=&\exp\left[-t\gamma\log(N)+\alpha(e^t-1)(2-\delta)\log(N)+\mathcal{O}(\log(N)^2/N)\right]\\
=&\exp\left[\log(N)(-t\gamma+\alpha(e^t-1)(2-\delta))
\right](1+o(1)).
\end{align*}
The function $-t\gamma+\alpha(e^t-1)(2-\delta)$ takes its minimum at $t=\log(\gamma/(\alpha(2-\delta)))$.

We aim at showing
$$
\mathbb{P}\left(\exists 1\leq i\leq N, T_i(\xi^1)\neq\xi_i^1\Big\vert\text{ } \Big\vert\log(N)-\sum_{j}\xi_j^1 \Big\vert<(1-\delta )\log(N)\right)\rightarrow0.
$$
Following the lines above, this probability can be estimated by
\begin{align*}
&\mathbb{P}\big(\exists 1\leq i\leq N, T_i(\xi^1)\neq\xi_i^1\big\vert \text{   } \big\vert\log(N)-\sum_{j}\xi_j^1
\big\vert<(1-\delta )\log(N)\big)\\
\leq&\mathbb{P}\big(\exists 1\leq i\leq N,\xi_i^1=0, T_i(\xi^1)\neq\xi_i^1\big\vert \sum_{j}\xi_j^1
=(2-\delta )\log(N)\big)\\
\leq &N\cdot \exp\left[\log(N)(-t\gamma+\alpha(e^t-1)(2-\delta))\right]\\
\leq&N\cdot \exp\left[\log(N)(-\gamma\log(\gamma/((2-\delta)\alpha))+\alpha(2-\delta)(\gamma/
(\alpha(2-\delta))-1))\right]\\
=&N\cdot \exp\left[\log(N)(-\gamma\log(\gamma/((2-\delta)\alpha))+\gamma-\alpha(2-\delta))\right]
\end{align*}
and we need
$$
\gamma\log(\gamma/((2-\delta)\alpha))-\gamma+\alpha(2-\delta)>1,
$$
which is fulfilled if
$$
\alpha<\frac{\gamma}{2-\delta}\frac{1}{e^{1+1/\gamma}}.
$$
So for each $\alpha<e^{-2}$, we can find a threshold $h=\gamma \log(N)$ such that
$$
\mathbb{P}(\exists 1\leq i\leq N, T_i(\xi^1)\neq\xi_i^1)\rightarrow 0.
$$
This proves the first part of the theorem.

\bigskip
For the second part notice that any fixed $\xi^\mu$ will have almost $\log N$ 1's such that we can delete $\rho \log N$ many of them and the statement of the theorem makes sense.
The rest of the proof of part two consists of choosing $h$ now as a value slightly smaller than $(1-\rho) \log N$ and repeating the above arguments. Indeed, call $\tilde \xi^1$ a configuration obtained from $\xi^1$ when deleting $\rho \log N$ 1's. Then, as above, the local field $S_i (\tilde \xi^1)$ splits into a signal term and a noise term:
$$
S_i(\tilde \xi^1)= \sum_{j \neq i} \tilde \xi^1_j J_{ij}=  \xi_i^1 \sum_{j \neq i} \tilde \xi^1_j + \sum_{j \neq i} \tilde \xi^1_j \sum_{\mu \ge 2} \xi_i^\mu \xi_j^\mu.
$$
In comparison to the first part of the proof the ingredient $\sum_{j \neq i} \tilde \xi^1_j$ of the signal term is decreased to a size of $(1-\rho)\log N$, while the noise term $\sum_{j \neq i} \tilde \xi^1_j \sum_{\mu \ge 2} \xi_i^\mu \xi_j^\mu$ is treated in a similar fashion as in part one and is typically of order $\alpha (1-\rho)(\log N)$.

\bigskip
For the third statement of the theorem we will make use of an observation that is also useful in the proof of Theorem \ref{willshaw2} and will actually be shown in this context:
For a message (without loss of generality $\xi^1$) with active neurons $\xi_1^1= \ldots = \xi_c^1=1$ and $\xi_i^1=0$ for all $i \ge c$ we show that for $M$ large enough, i.e. $M=\alpha N^2/ (\log N)^2$ and $\alpha >
-\log (1-e^{-1})$ with probability converging to 1, there exists an $i \ge c+1$ such that for all $j \le c$ there is a $\mu \ge 2$ such that $\xi_i^\mu \xi_j^\mu=1$.

After borrowing this statement from the proof of Theorem \ref{willshaw2} we can proceed as follows:
Taking into account that with overwhelming probability $c$ is larger than $(1-\delta)\log N$ for any $\delta >0$ and $N$ large enough, we see that in Amari's model for such an $i \ge c+1$
\begin{eqnarray*}
T_i(\xi^1) &=&  \Theta(\sum_{j \neq i} J_{ij}\xi^1_j-\gamma \log N)\\
          &=&  \Theta(\sum_{j \le c} J_{ij}-\gamma \log N) \\
           &\ge & \Theta((1-\delta) \log N-\gamma \log N)=1 \\
\end{eqnarray*}
if we choose $1-\delta >\gamma$. As we can choose $\delta>0$ arbitrarily small, with any threshold $\gamma \log N$ with $\gamma <1$ such a neuron will not be recovered correctly.
\end{proof}

\begin{proof}[Proof of Corollary \ref{willshaw1}]
The only thing one has to observe is that for each $i$ with $\xi_i^1=1$ we again have $T_i(\xi^1)=1$, because again $\sum_j\xi_j^1 \geq\gamma\log(N)$ {for any $\gamma <1$}.

On the other hand for {each} $i$ with $\xi_i^1=0$ we have that the probability that $\xi_i^1$ is turned into a 1 by the dynamics and thus not recovered correctly is given by $\P(\sum_j J_{ij} \xi^1_j \ge \gamma \log(N))$. Now,
$$
\sum_j J_{ij} \xi^1_j < \sum_j \xi^1_j\sum_{\mu \ge 2}  \xi_i^\mu \xi_j^\mu
$$
and the right hand side is the quantity considered in the previous proof. Thus the bound obtained {in the previous proof} is also a bound for the Willshaw model with threshold dynamics.
\end{proof}

\begin{rem}
\normalfont
Of course, the previous proof underestimates the storage capacity of the Willshaw model with threshold dynamics.
However, the difference between $J_{ij}$ and $\sum_{\mu \ge 2}  \xi_i^\mu \xi_j^\mu$ is not that huge. Indeed, for $M=\alpha N^2/(\log N)^2$ the latter is close to a Poisson random variable with parameter $\alpha$ and we will see in the next theorem, that even with a better performing dynamics we only reach a bound of $\alpha \le 0.45$.
\end{rem}

We continue with the Willshaw model with WTA dynamics.
\begin{proof}[Proof of Theorem \ref{willshaw2}]
We start with proving the third statement of the theorem. This will automatically yield the first part by setting $\rho$ to $0$.

Using the same method as in the proof of Theorem \ref{amari}, we can restrict the proof to the cases where ${c_1}$ neurons in the message $\xi^1$ are active, with $c_1 \in [(1- \varepsilon_1) c,  (1+ \varepsilon_1) c]$, for some small $\varepsilon_1>0$.
Assume that $f$ of the ``1''-bits in $\xi^1$ are erased and $k=c_1-f$ ``1''s are known.
Without loss of generality, we can assume that $\xi^1_i=1$ for $i\le c_1$ and $\xi^1_i=0$ for $i\ge c_1+1$.

Let $\tilde{\xi}^1\in \{0,1\}^N$ be a version of $\xi^1$ corrupted as described above, such that $\tilde{\xi}^1_i=1$ for $i\le k$ and $\tilde{\xi}^1_i=0$ for $i\ge k+1$.
We have trivially that,
$$
\displaystyle h_i(\tilde{\xi}^1)= \sum_{j=1}^k J_{ij},
$$
and thus $\displaystyle h_i(\tilde{\xi}^1)= k$ for all $i\le c_1$. Therefore $y=T(\tilde{\xi}^1)$ will satisfy $y_i=\xi^1_i$ for all $i\leq c_1$.

Thus recalling the WTA we see that $y \ne \xi^1$, if there exist $i\ge c_1+1$, such that for all $j \le k$ there exists $\mu\ge 2$ such that $\xi^\mu_i \xi^\mu_j=1$.

The probability of the latter event can be bounded as follows. Consider
\begin{eqnarray*}
&&\P[\exists i\ge c_1+1, \forall j\le k : \exists \mu\ge 2,  \xi_i^\mu \xi_j^\mu=1]\\
&\le&  N\sum_{l=0}^{M-1}  \sum_{\underset{ \mathrm{card}(I)=l}{ I\subset \{2,\ldots, M\} } } \P[\forall j\le k : \exists \mu\ge 2,  \xi_N^\mu \xi_j^\mu=1 \vert \xi_N^\mu=1 \Leftrightarrow\mu\in I]\P[ \xi_N^\mu=1 \Leftrightarrow\mu\in I]\\ &\le& N\sum_{l=0}^{M-1} \sum_I  \P[\forall j\le k : \exists \mu\in I,  \xi_j^\mu=1] \P[ \xi_N^\mu=1 \Leftrightarrow\mu\in I]\\
&=& N\sum_{l=0}^{M-1} \binom{M-1}{l}  ( 1- (1-\frac cN)^l)^k  (\frac cN)^l (1-(\frac cN))^{M-l-1}\\
&=&  N\sum_{l=0}^{M-1} \binom{M-1}{l}  \sum_{i=0}^k \binom{k}{i} (-1)^i (1-\frac cN)^{il} (\frac cN)^l (1-(\frac cN))^{M-l-1}\\
&=& N\sum_{i=0}^k \binom{k}{i} (-1)^i (1- \frac cN + \frac cN (1-\frac cN)^i)^{M-1}\\
\end{eqnarray*}
by elementary transformations.

Now we expand the term in the brackets and use the bound $1+x\le e^x$ for all $x$ to obtain
\begin{eqnarray*}
&&\P[\exists i\ge c_1+1, \forall j\le k : \exists \mu\ge 2,  \xi_i^\mu \xi_j^\mu=1]\\
&\le &N\sum_{i=0}^k \binom{k}{i} (-1)^i \left (1- i \left(\frac cN\right)^2+ \frac{i(i-1)}2 \left(\frac cN\right)^3+\mathcal{O}\left(i^3  \left(\frac cN\right)^4\right)\right)^{M-1}\\
&\leq& N\sum_{i=0}^k \binom{k}{i} (-1)^i \exp\left (- iM \left(\frac cN\right)^2+M \frac{i(i-1)}2 \left(\frac cN\right)^3+\mathcal{O}\left(M i^3  \left(\frac cN\right)^4\right)\right)\\
&=& N\sum_{i=0}^k \binom{k}{i} (-1)^i e^{- iM \left(\frac cN\right)^2}\left (1+M \frac{i(i-1)}2 \left(\frac cN\right)^3+\mathcal{O}\left(M i^3  \left(\frac cN\right)^4\right)\right)\\
&\le& N (1-e^{- M (\frac cN)^2})^k  +  MN(\frac cN)^3 \sum_{i=0}^k \binom{k}{i} (-1)^i e^{- iM (\frac cN)^2} \frac{i(i-1)}2\\
&& \qquad +N(1+e^{- M (\frac cN)^2})^k \mathcal{O}\left(M   k^3\left(\frac cN\right)^4\right)\\
\end{eqnarray*}
\begin{eqnarray*}
&=&N (1-e^{- M (\frac cN)^2})^k  +  MN(\frac cN)^3 e^{- 2M (\frac cN)^2}\frac{k(k-1)}2(1-e^{- M (\frac cN)^2})^{k-2}\\
&& \qquad +N(1+e^{- M (\frac cN)^2})^k \mathcal{O}\left(M   k^3\left(\frac cN\right)^4\right).
\end{eqnarray*}

If we choose
$$
M=\alpha (\frac Nc)^2  \quad \mbox{and  } k=(1-\rho) \log(N) \quad \mbox{for some } \rho\in [0,1[
$$
we arrive at
\begin{eqnarray*}
&&\P[\exists i\ge c_1+1, \forall j\le k : \exists \mu\ge 2,  \xi_i^\mu \xi_j^\mu=1] \\
&\le& N (1-e^{- \alpha})^{(1-\rho)  \log(N)}  +  \alpha  e^{- 2\alpha}(\log N)^3(1-e^{- \alpha})^{(1-\rho)  \log N-2}\\
&&\qquad +(1+e^{- \alpha})^{(1-\rho) \log(N)} \mathcal{O}(\frac {(\log N)^5} N).
\end{eqnarray*}
If $(1-\rho) \log(1-e^{- \alpha})<-1$, i.e.  $\alpha<-\log(1-e^{-1/(1-\rho) })$,  the first term converges to 0 and  the two last terms also vanish for $N\rightarrow\infty$. This gives
$$
\P[\exists i\ge c_1+1, \forall j\le k : \exists \mu\ge 2,  \xi_i^\mu \xi_j^\mu=1] \rightarrow 0
$$
as desired.

It remains to prove the reverse bound on the storage capacity. The considerations are similar to what we did above.
Now assume that $M \ge \alpha (\frac N c)^2$ for some $\alpha >0$ and again that $\xi^1$ has entries $\xi_i^1= 1$ for $i=1, \ldots c_1$ and  $\xi_i^1= 0$ for $i >c_1$.

Again consider
\begin{eqnarray*}
&&\P[\exists i\ge c_1+1, \forall j\le k : \exists \mu\ge 2,  \xi_i^\mu \xi_j^\mu=1]\\
&=& 1-\P\Big[\bigcap_{i \ge c_1+1} \{\exists j\le k : \forall \mu\ge 2,  \xi_i^\mu \xi_j^\mu=0\}\Big]\\
&=& 1-\P_{\{\xi_j^\mu, j \le k, \mu \ge 2\}} \prod_{i=c_1+1}^N  \P_{\{\xi_i^\mu, \mu \ge 2\}} \Big[\exists j\le k : \sum_{\mu\ge 2} \xi_i^\mu \xi_j^\mu=0\Big]\\ \end{eqnarray*}
by independence after conditioning (and the $\P_{\{\xi_j^\mu\}}$ denote the probabilities with respect to the corresponding random variables).
Now
\begin{eqnarray*}
&&\P_{\{\xi_i^\mu, \mu \ge 2\}} \Big[\exists j\le k : \sum_{\mu\ge 2} \xi_i^\mu \xi_j^\mu=0\Big]
= 1-\P_{\{\xi_i^\mu, \mu \ge 2\}} \Big[\forall j\le k : \sum_{\mu\ge 2} \xi_i^\mu \xi_j^\mu\ge 1\Big]
\end{eqnarray*}
Let $X_j :=\sum_{\mu \ge 2} \xi_i^\mu \xi_j^\mu$. We observe by similar arguments as in Section 3 that the $(X_j)$ are positively associated with respect to $\P_{\{\xi_i^\mu, \mu \ge 2\}}$.
Therefore, for $ i \ge c_1+1,$
$$
\P_{\{\xi_i^\mu, \mu \ge 2\}} [\forall j\le k : X_j\ge 1]\ge
\prod_{j=1}^k\left(\P_{\{\xi_i^\mu, \mu \ge 2\}} [X_j\ge 1]\right)
$$
which gives
\begin{eqnarray*}
&&\P[\exists i\ge c_1+1, \forall j\le k : \exists \mu\ge 2,  \xi_i^\mu \xi_j^\mu=1] \\ &\ge&  1-\P_{\{\xi_j^\mu, j \le k, \mu \ge 2\}} \prod_{i=c_1+1}^N  \left(1-\prod_{j=1}^k(\P_{\{\xi_i^\mu, \mu \ge 2\}} [X_j\ge 1])\right)
\end{eqnarray*}
To compute the right hand side take e.g. $i=N$. Then for all $j\le k,$
\begin{eqnarray*}
\P_{\{\xi_N^\mu, \mu \ge 2\}} [X_j\ge 1]) &=& 1-\P_{\{\xi_N^\mu, \mu \ge 2\}} \left(\sum_{\mu=1}^M \xi_N^\mu \xi_j^\mu =0\right)\\
&=&  1-\prod_{\mu: \xi_j^\mu=1} \P_{\{\xi_N^\mu, \mu \ge 2\}} \left(\xi_N^\mu=0\right)\\
&=& 1-\left(1-\frac c N\right)^{W_j},\\
\end{eqnarray*}
where $W_j:= \sum_{\mu=1}^M \xi_j^\mu.$
With overwhelming probability
$$
W_j \in \Big[(1-\vep) \frac{Mc}N,(1+\vep) \frac{Mc}N\Big]
$$
for all $N$ large enough, for all $j\le k$. More precisely, for all $\vep>0$, $k= C \log(N)$, with $C>0$,
$$
\P\Big[\forall j\le k : W_j\in \Big[(1-\vep) \frac{Mc}N,(1+\vep)\frac{Mc}N\Big]\Big] \ge 1-  2C \log(N) e^{-\frac{Mc\vep^2}{2N}}.
$$
This justifies that we can restrict to these cases, and putting things together, we obtain for $M = \alpha (\frac N c)^2$ that
\begin{eqnarray*}
&&\P[\exists i\ge c_1+1, \forall j\le k : \exists \mu\ge 2,  \xi_i^\mu \xi_j^\mu=1]
\ge 1-\left(1-\left(1-e^{-\alpha}\right)^k\right)^{N-k_1}.
\end{eqnarray*}
The right hand side converges to 1 if $\left(1-\left(1-e^{-\alpha}\right)^k\right)^N$ goes to 0, which is the case if and only if
\begin{eqnarray*}
N \log \left(1-\left(1-e^{-\alpha}\right)^k\right) &\approx & -N \left(1-e^{-\alpha}\right)^k\\
&=& -N \exp\left( k \log(1-e^{-\alpha})\right)\\
&=& -N^{1+(1-\rho) \log (1-e^{-\alpha})} \to -\infty.
\end{eqnarray*}
This is true if $1+(1-\rho) \log (1-e^{-\alpha}) >0$, which is true if and only if
$$
\alpha > -\log (1-e^{-1/(1-\rho) }).
$$
This finishes the proof.
\end{proof}

\begin{rem}
Note that the previous proof reveals that not only we have upper and lower bounds on the storage capacity of the Willshaw model with WTA dynamics, but also that these bounds match.
Such matching bounds can very rarely be proven. The only other model we are aware of where this is the case, is the Hopfield model (see \cite{MPRV} and \cite{Bov98}).
\end{rem}

\begin{proof}[Proof of Theorem \ref{GB-WTA}]
The decisive observation here is that the GB model is ``almost'' a Willshaw model. As a matter of fact, as stated already in the description of the model in Section 2,  the only difference is, that in the GB model there is a restriction on the location of the 1's. However, if we analyze the proof of Theorem \ref{willshaw2}, we find that the dynamics is in a sense ''non-spatial'', i.e. a neuron is getting signals from all the other neurons in both of these models. Thus this detail does not influence the proof.

This observation, however, also raises the question, whether we can also prove the third statement in Theorem \ref{willshaw2} for the GB model. It is, indeed, natural to conjecture that a similar statement holds true. However, in the proof of part 3 of Theorem \ref{willshaw2} we make use of positive association. This property enters the proof in the Willshaw model, because with our setting we are having increasing functions of i.i.d. random variables (the spins $\xi_i^\mu$), that are indeed positively associated. In the GB model, the extra condition that each pattern has exactly one 1 in each of the blocks implies that for each fixed $\mu$ the random variables $\xi_{(a,k)}^\mu$ are no longer independent. Hence we do not have positive association.  
%
%
% with an extra condition on the location of the 1's. This makes it possible to almost literally repeat the proof of Theorem \ref{willshaw2}. The details are left to the reader.
\end{proof}

%**

\section{The wrong message revisited -- a limit of all reconstruction techniques}
In this section we return to the question addressed in Section 3. There we showed that in the GB model with $M$ too large a wrong message will be recognized with large probability as a correct one, which limits the confidence we can have into our associative memory.

A very similar consideration shows that we cannot reconstruct erased messages in the GB model, if $M$ is too large. Indeed, in the GB model suppose we delete at random a proportion of $(1-\rho)c$ of active bits of a given message. If the remaining bits can be completed in more than one way to a message that is recognized by the system (N.B. not necessarily a message that is stored in the network), there is no way whatsoever, a reconstruction algorithm could find the correct message with probability one.

Using ideas from Section 3 one can prove a theorem on the probability to complete an erased message by a message on a given set of neurons. To formulate it, suppose that a message $\xi^1$ is stored in the network. Without loss of generality $\xi^1_{(a,1)}=1$ for all clusters $1 \le a \le c$ and all the other bits are 0. Assume we keep the $\xi^1_{(a,1)}=1$ for the clusters $1 \le a \le \rho c$, $0 <\rho < 1$ and set all other neurons to 0. Then for each cluster $\rho c+1 \le a \le c$ we choose a neuron $(a,i), 2\le i \le l$ and set it to 1. Let $G$ be the event that the message $\zeta$ having 1's in position $(a,1), 1 \le a \le \rho c$ and $(a,i)$ for $\rho c+1 \le a \le c$ is recognized by the system as a stored message.

\begin{theorem}  \label{subclique}
Suppose that in the GB model we store $M= \alpha l^2 \log c$ messages. Then  $\P(G)$ tends to 0 if and only if $\alpha<2$.
\end{theorem}

\begin{proof}
We only sketch the proof here as it is almost identical to the considerations in Section 3.

Other than there, we already know $\rho c$ bits of $\zeta$ are correct. Hence we only need to find messages that are active on the remaining $r(c,\rho):=\rho (1-\rho) c^2+(1-\rho)c ((1-\rho)c-1)/2 =\frac {c^2} 2(1-\rho^2)-\frac{1}{2}c(1-\rho)$ edges. % \notejud{I get a number of $\frac{c^2}{2}(1-\rho^2)-\frac{1}{2}c(1-\rho)$ edges.

Positive association bounds thus $\P(G)$ by $(1-(1-\frac 1 {l^2})^M)^{r(c,\rho)}=:d^{~r(c,\rho)}$ from below. The same exponential inequality as in Section 3 also shows an upper bound for $\P(G)$ by $d^{~r(c,\rho)}$ plus a vanishing term. Replacing $(1-\frac 1 {l^2})^M$ by $c^{-\alpha}$ we thus see that $d^{~r(c,\rho)}$ is of order $\exp(-\frac{c^{2-\alpha}}{2}(1-\rho^2))$ and therefore goes to zero, if and only if, $c^{2-\alpha}  (1-\rho^2) \to \infty$.
\end{proof}

\begin{rem}\normalfont
Similarly to Theorem \ref{Thsection3}, we get that $\P(G)$ is well approximated by $d^{\frac {c^2} 2(1-\rho^2)}$, when the latter goes to 0, for $\alpha\in ]1,2[$.
This is not the case for $\alpha\in ]0,1[$, since  the additive error term in the upper bound vanishes, but  slower than $d^{\frac {c^2} 2(1-\rho^2)}$.
\end{rem}

\section{Dynamical properties of the models}
An interesting question is the convergence of the proposed dynamics. Recall that we distinguish two types of dynamics: a) fixed threshold ones where $h$ is fixed a priori and b) varying threshold ones where $h$ is updated at each iteration of the dynamics (e.g. WTA). Note that in all cases we consider the memory effect described in Section 2.2.

Let us first consider the Willshaw model.

\subsection{Willshaw model}

In this section we show the following results:
\begin{enumerate}
\item Choosing a fixed $h$ forces convergence of the dynamics,
\item Choosing a varying $h$ can lead to oscillations in the dynamics,
\item Choosing the threshold $h_{(1)}$ as defined in Section 2, performance does not benefit from iterating more than once the dynamics.
\end{enumerate}

Note that the major interest of varying thresholds is that they lead to better performance as illustrated in Section~\ref{sec:simulations}. There thus exists a tradeoff between performance and convergence guarantees for the Willshaw model.

\begin{theorem}
  Choosing a fixed threshold $h$ forces the dynamics to converge.
  \label{thm:convergence}
\end{theorem}
\begin{proof}
Let us consider an input pattern $\tilde{\xi}^\mu$ where some 1s have been erased. Denote $c_\mu=\|\tilde{\xi}^\mu\|_0$ to be the number of 1's in $\tilde{\xi}^\mu$. Then it is immediate that if $h>c_\mu$ the dynamics converges in one iteration to a null vector.

On the other hand, let us introduce the sequence $\left(\tilde{\xi}^{\mu}(t)\right)_{t\geq 0}$:

\begin{eqnarray*}
\tilde{\xi}^\mu(0) &:=& \tilde{\xi}^\mu\\
 \tilde{\xi}^\mu(t+1) &:=& T\left(\tilde{\xi}^\mu(t)\right) \qquad \text{and for all } t \in \mathbb{N},
\end{eqnarray*}
and the sequence $\left(a^{\mu}(t)\right)_{t \geq 0}$ such that $a^\mu(t) = \{i, \tilde{\xi}^\mu_i(t) = 1\}$ for all $t \in \mathbb{N}_0$.

We now can show the following proposition:
\begin{proposition}
If $h\leq c_\mu$, the sequence $\left(a^{\mu}(t)\right)_{t \geq 0}$ is nondecreasing with respect to inclusion.
\end{proposition}
\begin{proof}
  Let us proceed by  induction.

  First we have trivially that $a^\mu(0)\subseteq a^\mu(1)$. This is due to the fact that $\forall i,j\in a^\mu(0), J_{ij}=1$.

  Then let us suppose that for some $t$ we have $a^\mu(t)\subseteq a^\mu(t+1)$. By definition, $\forall i\in a^\mu(t+1)$, we have $\#\{j \in a^\mu(t), J_{ij} = 1\} \geq h$, where $\#$ denotes the cardinality operator.

  Since $a^\mu(t)\subseteq a^\mu(t+1)$, it also holds that $\#\{j \in a^\mu(t+1), J_{ij} = 1\} \geq h$ and we conclude that $a^\mu(t+1)\subseteq a^\mu(t+2)$.
\end{proof}

A direct corollary is that $\left(a^{\mu}(t)\right)_{t \geq 0}$ converges.
\end{proof}

\begin{theorem}
  Choosing a varying $h$ can lead to oscillations in the dynamics of the Willshaw model.
\end{theorem}
\begin{proof}To illustrate this property, we propose an example where $N=5$ and $c=2$. We choose the threshold $h_{(1)}$ as defined in Section 2. Let us consider that:
\[
\left(\xi^\mu\right)_{1\leq \mu \leq 6} = \left(\left(\begin{array}{c}1\\1\\0\\0\\0\end{array}\right)
  \left(\begin{array}{c}1\\0\\1\\0\\0\end{array}\right)
    \left(\begin{array}{c}1\\0\\0\\1\\0\end{array}\right)
      \left(\begin{array}{c}0\\1\\0\\0\\1\end{array}\right)
        \left(\begin{array}{c}0\\0\\1\\0\\1\end{array}\right)
          \left(\begin{array}{c}0\\0\\0\\1\\1\end{array}\right)\right).
            \]

Consider the input:
\[
\tilde{\xi}^\mu(0) = \left(\begin{array}{c}1\\0\\0\\0\\0\end{array}\right)\;.
\]

One can easily check that:
\[
\left(\tilde{\xi}^\mu(t)\right)_{0\leq t \leq 3} = \left(\left(\begin{array}{c}1\\0\\0\\0\\0\end{array}\right)
  \left(\begin{array}{c}1\\1\\1\\1\\0\end{array}\right)
    \left(\begin{array}{c}1\\0\\0\\0\\0\end{array}\right)
      \left(\begin{array}{c}1\\1\\1\\1\\0\end{array}\right)
\right)\;,
\]
and thus $\tilde{\xi}^\mu(2) = \tilde{\xi}^\mu(0)$.
\end{proof}

The same dynamics is illustrated in Figure~\ref{fig:convergence}.

\begin{figure}
  \begin{centering}
    \begin{tikzpicture}
      \tikzstyle{node} = [draw,circle];
      \node at (1,2) {$t = 0$};
      \node[node](1)[fill] at (0,0) {};
      \node[node](2) at (1,1) {};
      \node[node](3) at (1,0) {};
      \node[node](4) at (1,-1) {};
      \node[node](5) at (2,0) {};
      \path
      (1) edge (2)
      edge (3)
      edge (4)
      (5) edge (2)
      edge (3)
      edge (4)
      (1) edge[loop left] (1)
      (2) edge[loop above] (2)
      (3) edge[loop above] (3)
      (4) edge[loop below] (4)
      (5) edge[loop right] (5);
      \begin{scope}[xshift=4cm]
        \node at (1,2) {$t = 1$};
        \node[node](1)[fill] at (0,0) {};
        \node[node](2)[fill] at (1,1) {};
        \node[node](3)[fill] at (1,0) {};
        \node[node](4)[fill] at (1,-1) {};
        \node[node](5) at (2,0) {};
        \path
        (1) edge (2)
        edge (3)
        edge (4)
        (5) edge (2)
        edge (3)
        edge (4)
        (1) edge[loop left] (1)
        (2) edge[loop above] (2)
        (3) edge[loop above] (3)
        (4) edge[loop below] (4)
        (5) edge[loop right] (5);
      \end{scope}
      \begin{scope}[xshift=8cm]
        \node at (1,2) {$t = 2k, k\geq 1$};
        \node[node](1)[fill] at (0,0) {};
        \node[node](2) at (1,1) {};
        \node[node](3) at (1,0) {};
        \node[node](4) at (1,-1) {};
        \node[node](5) at (2,0) {};
        \path
        (1) edge (2)
        edge (3)
        edge (4)
        (5) edge (2)
        edge (3)
        edge (4)
        (1) edge[loop left] (1)
        (2) edge[loop above] (2)
        (3) edge[loop above] (3)
        (4) edge[loop below] (4)
        (5) edge[loop right] (5);
      \end{scope}
      \begin{scope}[xshift=12cm]
        \node at (1,2) {$t = 2k+1,k\geq 1$};
        \node[node](1)[fill] at (0,0) {};
        \node[node](2)[fill] at (1,1) {};
        \node[node](3)[fill] at (1,0) {};
        \node[node](4)[fill] at (1,-1) {};
        \node[node](5) at (2,0) {};
        \path
        (1) edge (2)
        edge (3)
        edge (4)
        (5) edge (2)
        edge (3)
        edge (4)
        (1) edge[loop left] (1)
        (2) edge[loop above] (2)
        (3) edge[loop above] (3)
        (4) edge[loop below] (4)
        (5) edge[loop right] (5);
      \end{scope}
    \end{tikzpicture}
  \end{centering}
  \caption{Illustration of the oscillation of the dynamics when using WTA with the Willshaw model. Here the model contains $N=5$ neurons and the number of 1s in stored messages is $c=2$.}
  \label{fig:convergence}
\end{figure}
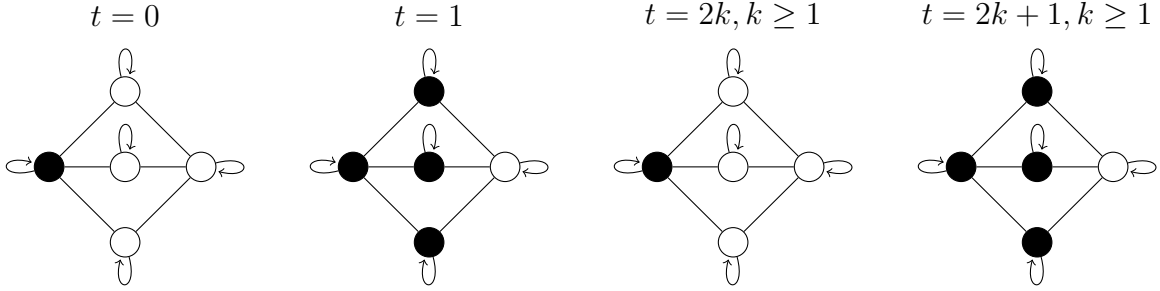

More generally, using the threshold $h_{(1)}$ as defined in Section 2, the performance of the model does not benefit from using more than one iteration, as expressed in the following theorem:
\begin{theorem}
  Consider a Willshaw network where the threshold is chosen as the maximum of the achieved scores ($h_{(1)}$). Choose as input a partially (but not completely) erased version $\tilde{\xi}^\mu$ of a stored message $\xi^\mu$. Then the dynamics converges  if and only if it converges in one step. In particular, it can only converge to $\xi^\mu$ if it does so in one iteration.
  \label{thm:stuck1iteration}
\end{theorem}
\begin{proof}
  Let us use the same notations as in the proof of Theorem~\ref{thm:convergence}. We denote by $h_{(1)}(t)$ the value of the threshold at step $t$.

  Let us discuss two cases:
  \begin{enumerate}
  \item In the first case it holds for all $i$ and $j$ that $ \tilde{\xi}_i^\mu(1) = 1$ and $\tilde{\xi}_j^\mu(1) = 1$ implies that  $J_{ij} = 1$. In other words:  All activated neurons after one iteration are connected one to another. In this case one can easily check that $$h_{(1)}(1) = \mathrm{card}(\{i, \tilde{\xi}_i^\mu(1) = 1\})$$
and thus we have for all $t \geq 1$ that $\tilde{\xi}^\mu(1) = \tilde{\xi}^\mu(t)$.
\item There are $i'$ and $j'$ such that  $\tilde{\xi}_{i'}^\mu(1) = 1$ and $\tilde{\xi}_{j'}^\mu(1) = 1$ but $ J_{{i'}{j'}} = 0$, i.e. there are activated neurons that are not interconnected.

\noindent
Note that by construction of $J$ we then cannot have that $\tilde{\xi}^\mu(1) = \xi^\mu$. We fix such a pair  $i'$ and $j'$. By construction of $J$, we have for  all $i$ and $j$ that $\xi^\mu_i = 1$ and $\xi^\mu_j = 1$ implies that $ J_{ij} = 1$. As a direct consequence, we obtain that
$$
h_{(1)}(0) = \mathrm{card}(\{i, \tilde{\xi}_i^\mu(0) = 1\})
$$
and therefore all neurons activated at step 0 are connected to all neurons activated at step 1 (note also that $\{i, \xi^\mu_i = 1\} \subsetneq \{i, \tilde{\xi}_i^\mu(1) = 1\}$). Thus we obtain
$$
h_{(1)}(1) = \mathrm{card}(\{i, \tilde{\xi}_i^\mu(1) = 1\}).
$$
Consequently, $\tilde{\xi}_{i'}^\mu(2) = 0$ and $\tilde{\xi}_{j'}^\mu(2) = 0$ which leads to $\tilde{\xi}^\mu(1) \not= \tilde{\xi}^\mu(2)$. We conclude that the neurons activated in $\tilde{\xi}^\mu(2)$ are those connected to all neurons in $\tilde{\xi}^\mu(1)$. In particular we obtain $\{i, \tilde{\xi}_i^\mu(0) = 1\} \subseteq \{i, \tilde{\xi}_i^\mu(2) = 1\}$.

\noindent
Similarly we have that $h_{(1)}(2) = \mathrm{card}(\{i, \tilde{\xi}_i^\mu(2) = 1\})$.

We then observe that there cannot be a neuron active at step 3 that is not active at step 1, as the neurons activated at step 3 are connected to all neurons activated at step 2 and thus to all neurons activated at step 0. We conclude that for all $t\geq 1$ we have that
$$
\tilde{\xi}^\mu(2t-1) = \tilde{\xi}^\mu(2t+1)  \qquad \mbox{and   }\tilde{\xi}^\mu(2t) = \tilde{\xi}^\mu(2t+2),
$$
together with $\tilde{\xi}^\mu(1) \not= \tilde{\xi}^\mu(2)$.
  \end{enumerate}
\end{proof}

\subsection{GB model}

Interestingly, the specific GB structure can be exploited in order to provide good performance and to ensure at the same time convergence of the dynamics. This is thanks to the previously mentioned SUM-OF-MAX rule (see Equation \eqref{eq:som}).
Recall the SUM-OF-MAX dynamic rule:
$$
T_{(a,k)}(\sigma) =  \Theta(s_{(a,k)}(\sigma)- h(a)),\ \mbox{where}\ h(a)=\max\{s(a,k), k=1,\ldots,l\}.
$$

This rule can be advantageously combined with a modification of the input when
retrieving a partially erased image. This modification consists in activating all
neurons in clusters where no neuron is active. Then we have trivially $h(a) = c$ for all $a$ and
this modification is such that the set of active neurons is non-increasing with
iterations of the dynamics.

Here is a rapid sketch of the proof of this result: to
be activated using the SUM-OF-MAX rule, a neuron has to be connected to at
least one activated neuron in each cluster. In particular it has to be connected to
an activated neuron in its own cluster. Due to the specific structure of the GB
model, the only connection a neuron may have with a neuron in its own cluster
is with itself. Therefore, to be activated, a neuron has to already be activated at
the previous step of the dynamics.

We refer to this algorithm as ``SOM'' in Figure~\ref{fig:performanceofnetworks}.

\subsection{Simulations}
\label{sec:simulations}

In order to compare the performance of the three above mentioned solutions, we run several simulations. We consider that the number of 1s in each message is $c$ for the Willshaw model.

We propose to use three different families of algorithms: a) fixed threshold ones where $h$ is determined a priori, b) varying threshold ones where $h$ can be modified at each iteration and c) exhaustive search where solutions are looked for using a brute-force approach. This last option allows us to compare the different models intrinsically, thus removing any bias from chosen retrieval dynamics.

For case a) we define $h$ as the number of 1s in the input pattern. This value appears to be optimal for most cases we simulated.
For case b) we use the winner-takes-all algorithm previously described in which we select $h$ so that the number of 1s in the obtained vector is minimum and at least $c$.
For case c) we use an exhaustive search of potential candidates and select randomly one of them.  Note that for Amari's model we select the clique (or one of the cliques) that achieve the maximum sum of inner edge weights.
Finally, for each case we also plot the obtained curves when using SUM-OF-MAX with the GB model for easier comparison of performance.

We depict the evolution of the error rate for a given problem as a function of the number of stored patterns. This measure is not totally fair as:
\begin{itemize}
\item A stored pattern with $c$ 1s using the Willshaw model or Amari's one made of $N$ neurons has entropy $\log_2\left(\binom{N}{c}\right)$ whereas with the GB model its entropy is lesser: $C\log_2\left(l\right)$.
\item The number of possible connections in a Willshaw model or Amari's one with $N$ neurons is larger than that using a GB model with the same number of neurons. Moreover in the Amari model each connection can take up to $M$ distinct values.
\end{itemize}

In order to account for these differences, we propose to depict also the evolution of the error rate as a function of the efficiency of the model, defined as the ratio between the entropy of the set of stored patterns and the number $C$ of bits required for straightforward encoding of the used synaptic weights. The latter value $C$ depends on the model parameters: for an Amari model made of $N$ neurons and storing $M$ patterns, it is equal to:
\[
C_{Amari} = \binom{N}{2} \log_2(M+1)\;.
\]
For the Willshaw model it becomes:
\[
C_{Willshaw} = \binom{N}{2}\;.
\]
For the GB model, it depends on the parameters $c$ and $l$ and becomes:
\[
C_{GB} = \binom{c}{2}l^2\;.
\]

The results are depicted in Figure~\ref{fig:performanceofnetworks}. Some remarks about these results:

\begin{figure}
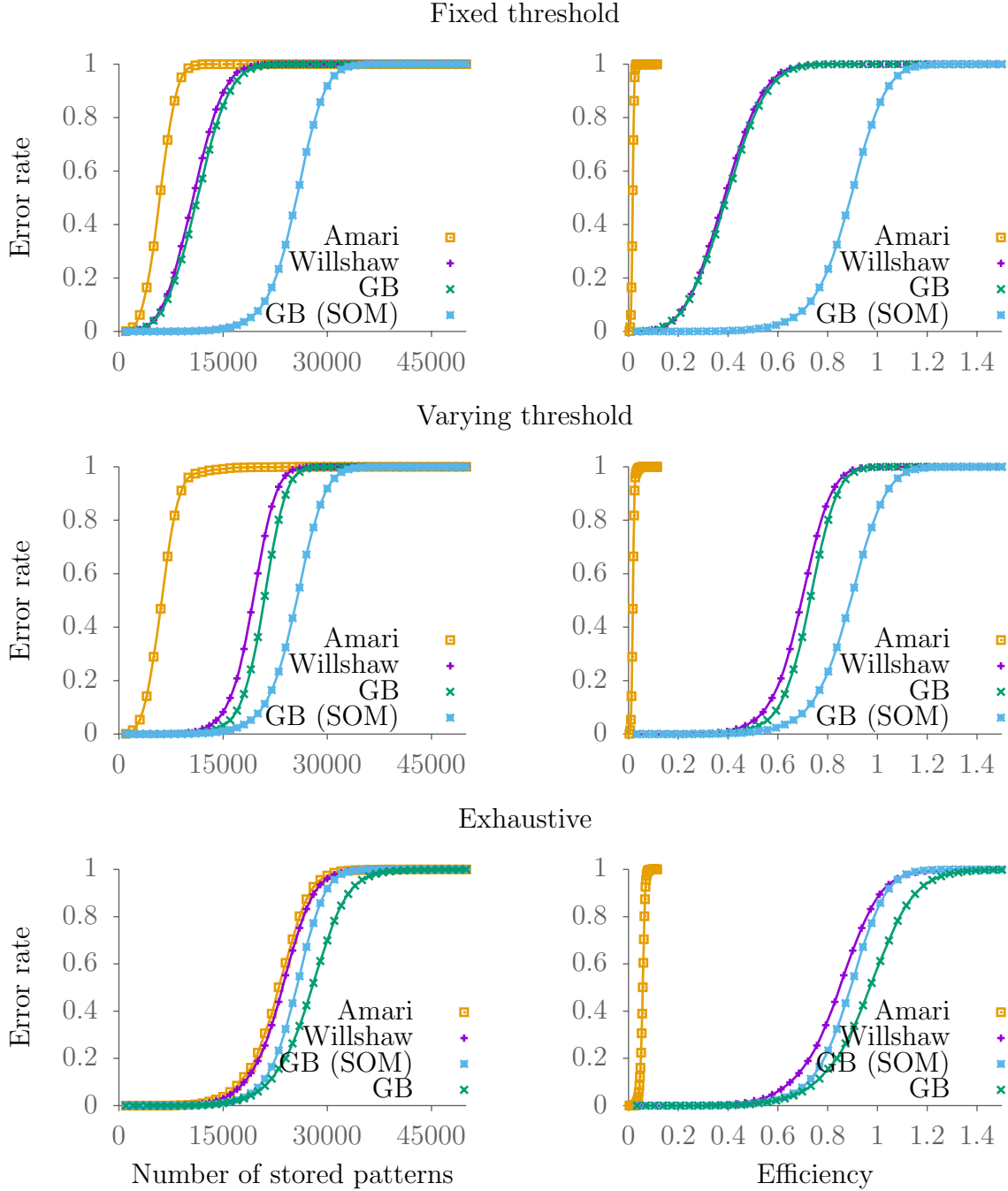

  \begin{centering}
    \begin{tabular}{cc}
      \multicolumn{2}{c}{Fixed threshold}\\
      \input{out-fixed-messages.tex}&\input{out-fixed.tex}\\
      \multicolumn{2}{c}{Varying threshold}\\
      \input{out-varying-messages.tex}&\input{out-varying.tex}\\
      \multicolumn{2}{c}{Exhaustive}\\
      \input{out-exhaustive-messages.tex}&\input{out-exhaustive.tex}
    \end{tabular}
  \end{centering}

  \caption{Comparison of performance of Amari, Willshaw and GB models (with proposed dynamics and SUM-OF-MAX (SOM)). For all simulated point, there are $N=2048$ neurons (grouped in $c=8$ clusters of $l=256$ neurons for the GB model), stored messages contain exactly $c=8$ 1s each and the objective is to retrieve a previously stored pattern when 4 out of the initial 8 1s in stored messages are missing. Each point is the average of 100.000 tests. Figures in first column depict the evolution of the error rate as a function of the number of stored patterns. Figures in second column depicts the evolution of the error rate as a function of efficiency. First line correspond to fixed threshold dynamics, second line to varying threshold strategies and third line to exhaustive ones.}
  \label{fig:performanceofnetworks}
\end{figure}

\begin{itemize}
\item No matter what algorithms are used, the performance of the methods clearly indicates that GB performs better than Willshaw that performs itself better than Amari's networks.
\item The only difference between Amari's networks and Willshaw's is the fact the former use weighted connections instead of binary ones. Our simulations clearly indicates that weights offer no gain in performance.
\item It appears clearly that fixed threshold algorithms perform worse than varying threshold ones.
\end{itemize}

%\newpage

 \bibliographystyle{abbrv}
\bibliography{LiteraturDatenbank}

\begin{thebibliography}{10}

\bibitem{griponc}
B.~K. Aliabadi, C.~Berrou, V.~Gripon, and X.~Jiang.
\newblock Storing sparse messages in networks of neural cliques.
\newblock {\em IEEE Transactions on Neural Networks and Learning Systems},
  25:980--989, 2014.

\bibitem{Bolle}
D.~Boll\'e and T.~Verbeiren.
\newblock Thermodynamics of fully connected {B}lume-{E}mery-{G}riffiths neural
  networks.
\newblock {\em J. Phys. A: Math. Gen}, 36(6):295--305, 2003.

\bibitem{Boutsikas2000}
M.~V. Boutsikas and M.~V. Koutras.
\newblock A bound for the distribution of the sum of discrete associated or
  negatively associated random variables.
\newblock {\em Ann. Appl. Probab.}, 10(4):1137--1150, 2000.

\bibitem{Bov98}
A.~Bovier.
\newblock Sharp upper bounds on perfect retrieval in the {H}opfield model.
\newblock {\em J. Appl. Probab.}, 36(3):941--950, 1999.

\bibitem{Burshtein}
D.~Burshtein.
\newblock Nondirect convergence radius and number of iterations of the
  {H}opfield associative memory.
\newblock {\em IEEE Trans. Inform. Theory}, 40(3):838--847, 1994.

\bibitem{Esary/Proschan/Malkup:1967}
J.~D. Esary, F.~Proschan, and D.~W. Walkup.
\newblock Association of random variables, with applications.
\newblock {\em Ann. Math. Statist.}, 38:1466--1474, 1967.

\bibitem{griponb}
V.~Gripon and C.~Berrou.
\newblock Sparse neural networks with large learning diversity.
\newblock {\em IEEE Transactions on Neural Networks}, 22(7):1087--1096, July
  2011.

\bibitem{HLV15}
J.~Heusel, M.~L{\"o}we, and F.~Vermet.
\newblock On the capacity of an associative memory model based on neural
  cliques.
\newblock {\em Statist. Probab. Lett.}, 106:256--261, 2015.

\bibitem{Hopfield1982}
J.~J. Hopfield.
\newblock Neural networks and physical systems with emergent collective
  computational abilities.
\newblock {\em Proc. Nat. Acad. Sci. U.S.A.}, 79(8):2554--2558, 1982.

\bibitem{Amari1989}
S.~ichi Amari.
\newblock Characteristics of sparsely encoded associative memory.
\newblock {\em Neural Networks}, 2(6):451 -- 457, 1989.

\bibitem{gripone}
H.~Jarollahi, V.~Gripon, N.~Onizawa, and W.~J. Gross.
\newblock Algorithm and architecture for a low-power content-addressable memory
  based on sparse-clustered networks.
\newblock {\em Transactions on Very Large Scale Integration Systems}, PP:1,
  October 2014.

\bibitem{gripond}
H.~Jarollahi, N.~Onizawa, V.~Gripon, and W.~J. Gross.
\newblock Algorithm and architecture of fully-parallel associative memories
  based on sparse clustered networks.
\newblock {\em Journal of Signal Processing Systems}, pages 1--13, 2014.

\bibitem{gripong}
H.~Jarollahi, N.~Onizawa, V.~Gripon, and W.~J. Gross.
\newblock Algorithm and architecture of fully-parallel associative memories
  based on sparse clustered networks.
\newblock {\em Journal of Signal Processing Systems}, pages 1--13, 2014.

\bibitem{griponf}
X.~Jiang, V.~Gripon, C.~Berrou, and M.~Rabbat.
\newblock Storing sequences in binary tournament-based neural networks.
\newblock {\em IEEE Transactions on Neural Networks and Learning Systems}, July
  2014.
\newblock Submitted to.

\bibitem{L98}
M.~L{\"o}we.
\newblock On the storage capacity of {H}opfield models with correlated
  patterns.
\newblock {\em Ann. Appl. Probab.}, 8(4):1216--1250, 1998.

\bibitem{Lo_biased}
M.~L{\"o}we.
\newblock On the storage capacity of the {H}opfield model with biased patterns.
\newblock {\em IEEE Trans. Inform. Theory}, 45(1):314--318, 1999.

\bibitem{LV_BEG}
M.~L{\"o}we and F.~Vermet.
\newblock The storage capacity of the {B}lume-{E}mery-{G}riffiths neural
  network.
\newblock {\em J. Phys. A}, 38(16):3483--3503, 2005.

\bibitem{LV07}
M.~L{\"o}we and F.~Vermet.
\newblock The capacity of {$q$}-state {P}otts neural networks with parallel
  retrieval dynamics.
\newblock {\em Statist. Probab. Lett.}, 77(14):1505--1514, 2007.

\bibitem{LV15}
M.~L{\"o}we and F.~Vermet.
\newblock Capacity of an associative memory model on random graph
  architectures.
\newblock {\em Bernoulli}, 21(3):1884--1910, 2015.

\bibitem{MPRV}
R.~J. McEliece, E.~C. Posner, E.~R. Rodemich, and S.~S. Venkatesh.
\newblock The capacity of the {H}opfield associative memory.
\newblock {\em IEEE Trans. Inform. Theory}, 33(4):461--482, 1987.

\bibitem{Okada1996}
M.~Okada.
\newblock Notions of associative memory and sparse coding.
\newblock {\em Neural Networks}, 9(8):1429 -- 1458, 1996.
\newblock Four Major Hypotheses in Neuroscience.

\bibitem{Palm1980}
G.~Palm.
\newblock On associative memory.
\newblock {\em Biological Cybernetics}, 36(1):19--31, 1980.

\bibitem{Palm2013}
G.~Palm.
\newblock Neural associative memories and sparse coding.
\newblock {\em Neural Networks}, 37(0):165 -- 171, 2013.
\newblock Twenty-fifth Anniversay Commemorative Issue.

\bibitem{Palm1996}
F.~Schwenker, F.~Sommer, and G.~Palm.
\newblock Iterative retrieval of sparsely coded associative memory patterns.
\newblock {\em Neural Networks}, 9(3):445 -- 455, 1996.

\bibitem{Willshaw}
D.~J. {Willshaw}, O.~P. {Buneman}, and H.~C. {Longuet-Higgins}.
\newblock {Non-Holographic Associative Memory}.
\newblock {\em Nature}, 222:960--962, June 1969.

\bibitem{griponh}
Z.~Yao, V.~Gripon, and M.~G. Rabbat.
\newblock A massively parallel associative memory based on sparse neural
  networks.
\newblock {\em Transactions on Parallel and Distributed Systems}, December
  2013.
\newblock Submitted to.

\end{thebibliography}
\end{document}